\RequirePackage[warning,log]{snapshot}
\documentclass{amsart}
\usepackage{graphicx}
\usepackage{amssymb}
\usepackage{amsmath}

\newtheorem{theorem}{Theorem}[section]
\newtheorem{lemma}[theorem]{Lemma}

\theoremstyle{definition}

\theoremstyle{remark}

\numberwithin{equation}{section}

\begin{document}

\title{A p-adic identity for Wieferich primes}

\author{Kok Seng Chua}
\email{chuakkss52@outlook.com}





\begin{abstract}Let $n$ be a positive integer, $p$ be an odd prime and integers $a,b \not= 0$  with $gcd(a,b)=1$, $p \nmid ab$, and $p|(a^n \pm b^n)$, we prove the identity
$$\nu_p(a^n \pm b^n)-\nu_p(n)=\nu_p(a^{p-1}-b^{p-1}).$$
An unintended interesting immediate consequence is the following  variant of Wieferich's criterion for FLT : Let $x^n+y^n=z^n$ with $n$ prime and  $x,y,z$ pairwise relatively prime. Then every odd prime $p|y$ satisfies $\nu_p(z^{p-1}-x^{p-1}) \ge n-1$ and every odd prime $p|x$ satisfies $\nu_p(z^{p-1}-y^{p-1}) \ge n-1$, and every odd prime $p|z$ satisfies $\nu_p(x^{p-1}-y^{p-1}) \ge n-1$, ie. every odd prime dividing $xyz$ is a   Wieferich prime of order at least $n-1$ to some base pair. In the "first case" where $n \nmid xyz$, the lower bound for the Wieferich order can be improved to $n$. This gives us very strong intuition why FLT should be true even for moderately large $n$.
\end{abstract}

\maketitle

\section{Motivation and result} For a prime $p$ and integer $n$, we shall denote the $p$-adic valuation by  $\nu_p(n)$, the exact power of $p$  dividing $n$. Since $(a^n-1,1,a^n)$ is an abc tuple, a very weak form of the ABC conjecture \cite{M,O} implies that the average power of  $(a^n-1)a^n$, given by  $\frac{\log((a^n-1).a^n)}{\log(rad((a^n-1).a^n))}$  has to be bounded above which means for $n$ large, $a^n-1$ must have many factor occurring to the first power to compensate for $a^n$ \cite{L} so that primes dividing $a^n-1$ to high powers must be rare.

This prompts us to look at primes  which divides $(2^n-1)$ to high powers, and we immediately observed  that $p^2|(2^n-1)$ implies $p|n$ and more over the defect $\delta:=\nu_p(2^n-1)-\nu_p(n)=1$ always for the first few hundreds $n$. This seems a little too good to be true and in fact it obviously can't hold for Wieferich prime where $p^2|(2^{p-1}-1)$. Indeed the first failure occur at $n=364=ord_p(2)$, the multiplicative order of $2$ mod the first Wieferich prime $p=1093$. Further experimentations with other bases reveal that  failure can only occur for Wieferich primes and they can only occur when the defect $\delta >1$. Also  the defect does not depend on $n$  for a fixed prime, Wieferich or not.
This leads us to the following identity
\begin{theorem} Let $a \not=0,1$ be an integer, $n$ a positive integer, and  $p$  an odd prime $p \nmid a$, and $p \mid (a^n \pm 1)$, then we have
\begin{align}
\nu_p(a^n \pm 1)-\nu_p(n)=\nu_p(a^{p-1}-1).
\end{align}
The minus identity still holds for $p=2$  except when $n$ is even and $a$ is $3$ mod $4$, the RHS should be replaced with $\nu_2(a+1)$.
The plus identity for $p=2$ only holds for $n$ odd as  $\nu_2(a^n+1)=\nu_2(a+1)$.
\end{theorem}
An immediate consequence, used in the proof that ABC implies existence of infinitely many non Wieferich primes \cite{Si}, which we were not thinking about, is the well known fact that  $\nu_p(a^n-1)=1$ implies  $p$ is non Wieferich, and we now know additionally that $\nu_p(n)=0$.
 Essentially the standard known proof of this special case also give us the proof of the full case, once we know which terms to keep track
 from the identity.
\begin{proof} Assume first $p|(a^n-1)$ and let $d=ord_p(a)$ we then have $d|p-1$ and $d|n$, so that $n=dr, p-1=dm$ and $p \nmid dm$ and hence $\nu_p(n)=\nu_p(r)$.

Let $\phi_p(x)=\sum_{j=0}^{p-1}x^j$ be the $p$th cyclotomic polynomial. We note that if $k$ is any nonzero integer
$\phi_p(1+kp)=\sum_{j=0}^{p-1}(1+kp)^j = \sum_{j=0}^{p-1} (1+jkp +O(p^2))=p+(p-1)pkp/2+O(p^2)=p \;(mod \; p^2)$, so $\nu_p(\phi_p(1+kp))=1$.

We will prove the minus identity by induction on $\nu_p(n)$.
Assume first $p \nmid n$ so $p \nmid r$, and $a^{p-1}-1=a^{dm}-1=(a^d-1)t$, where $t=\sum_{j=0}^{m-1}(a^d)^j = m \not=0 $ (mod $p$),
so  $\nu_p(a^{p-1}-1)=\nu_p(a^d-1)$.  Similarly $a^n-1=a^{dr}-1=(a^d-1)t$
where $t=\sum_{j=0}^{r-1}(a^d)^j = r  \not= 0$ \; (mod $p$). So $\nu_p(a^n-1)=\nu_p(a^d-1)=\nu_p(a^{p-1}-1)$.

Now let $m \ge 0$ and $n=dp^{m+1}r$ where $p \nmid r$, and let $u=dp^mr$, then since $p|a^n-1),a^{up}-1 = (a^u-1)=0 \; (mod \; p)$,  so
$\nu_p(\phi_p(a^u))=1$. We now have  $a^n-1=a^{up}-1=(a^u-1)\phi_p(a^u),$
so $\nu_p(a^n-1)=\nu_p(a^u-1)+1=\nu_p(a^{p-1}-1)+\nu_p(u)+1=\nu_p(a^{p-1}-1)+\nu_p(n)$. This proves the minus identity.

If $p|(a^n+1)$, then $p|(a^{2n}-1)$ and $p \nmid (a^n-1)$
so that $\nu_p(a^{2n}-1)=\nu_p(a^n+1)$, and  applying the minus identity gives
$\nu_p(a^n+1)-\nu_p(n)=\nu_p(a^{p-1}-1).$
\newline

  For $p=2$, $a$ is odd,  $(a^n-1)=(a-1)t$ where $t=\sum_{j=0}^{n-1} a^j=n \not=0$ \; $(mod \; 2)$, if $n$ is odd, so
  $\nu_2(a^n-1)=\nu_2(a-1)$. If $n$ is even, let $n=2^um$, $ u \ge 1$ and $m$ odd, then $a^n-1=\prod_{j=0}^{u-1}(a^{2^jm}+1)(a^m-1).$
  If $a=1$ mod $4$, $\nu_2(a^n-1)=u+\nu_2(a^m-1)=\nu_2(n)+\nu_2(a-1)$ where we use the odd case in the last term.
  If $a=3$ mod $4$, all term in the product contribute a single factor of $2$ except  $a^m+1$, so $\nu_2(a^n-1)=\nu_2(a^m+1)+\nu_2(n)$. But as before $a^m+1=(a+1)t$ where $t=\sum_{j=0}^{m-1} \pm a^j=m \not=0$ mod $2$, so $\nu_2(a^m+1)=\nu_2(a+1)$.

  For the plus cas for $p=2$, we assume $n$ is odd, then $(a^n+1)=(a+1)t$ where $t$ is odd as before. So $\nu_2(a^n+1)+\nu_2(a+1)$.
QED
\end{proof}
The identity in Theorem 1.1 is apparently new and it has many immediate consequences. The surprised  term  $\nu_p(n)$ is  seemingly the reason why there is a separation into case I and II for FLT.

Our original observation  that if $p^2|(a^n-1)$ or $p^2|(a^n+1)$, then $p|n$ unless we have the rare case that $p$ is Wieferich,  generalizes the speciaal case for Fermat and Mesernne numbers, for example \cite{P}. Also we have  $\nu_p(a^{ord_p(a)} \pm 1)=\nu_p(a^{p-1}-1)$ if $p \nmid a$. Perhaps the most interesting and unintended  immediate consequence is a considerable strengthening of Weiferich'c criterion for FLT but we need to generalize Wieferich prime to rational or rather an 2-tuple integral base.

\subsection{Extension to rational base and a strengthened Wieferich criterion for FLT}

The proof of Theorem 1.1 still works if $a \not= 0,1$ is  rational as was done in \cite{Si}, if we use the natural extension of the valuation $\nu_p(a/b):=\nu_p(a)-\nu_p(b)$. It is simpler to write in homogenous integral coordinates. Homogenising theorem 1.1 for rational $a/b$ gives us
\begin{theorem}
Let $a \not= b$ be nonzero integers  and $gcd(a,b)=1$ and $n$ a positive integer. If $p \nmid ab$,  and $p| (a^n \pm b^n)$, we have
\begin{align} \nu_p(a^n \pm b^n)-\nu_p(n)= \nu_p(a^{p-1}-b^{p-1}).
\end{align}
Again the same formula holds in the minus case for $p=2$ except when $n$ is even and $ab=3$ mod $4$ where the RHS should be replaced with $\nu_2(a+b)$. The plus identity holds only for $p=2$ and $n$ odd as $\nu_2(a^n+b^n)=\nu_2(a+b)$.
\end{theorem}
We still have the original observation that if $p^2|(a^n \pm b^n)$  then $p|n$ unless $p$ is Wieferich to the base $\{a,b\}$ (see below).

Even the case $n=1$ seems non obvious and we record it as
\begin{lemma} If a prime $p \ge 2$ is such that $p \nmid ab$ but $p|(a \pm b)$, then $$\nu_p(a^{p-1} - b^{p-1})=\nu_p(a \pm b),$$
except when $p=2$, only the minus equality hold. In particular if $p \nmid ab$ and $p|(a \pm b)$, then for all $n \ge 1$, we have

\begin{align} \nu_p(a^n \pm b^n)=\nu_p(n) +\nu_p(a \pm b). \end{align}
\end{lemma}
\begin{proof} First equality follow from setting $n=1$ in (1.2). The second follows from $p|(a \pm b)$ implies $p|(a^n \pm b^n)$.
\end{proof}
Note (1.3) does not hold if we only know $p|(a^n \pm b^n)$ , for example,
 $3|(5^4-4^4)$ but $2=\nu_3(5^4-4^4) \not= \nu_3(4) +\nu_3(5-4)=0$.
\newline

We now note that the RHS of (1.2), $\nu_p(a^{p-1}-b^{p-1}) \ge 1$ always, being a homogenization of little Fermat $\nu_2(a^{p-1}-1) \ge 1$. This  suggests the homogenization of Weiferich prime to a pair of bases. We say that a prime $p \nmid ab$ is a Wieferich prime to the base $(a,b)$ of order $r$ if $\nu_p(a^{p-1}-b^{p-1})=r > 1$. Clearly the usual Wieferich  prime to the base $a$ is Wieferich to the base $(1,a)$ and being Wieferich to the base $(a,b)$ is the same as to the base $(b,a)$ so we shall always assume $a<b$. Probabilistically, being Wieferich to a base $(a,b)$ is as rare as being Wieferich to a base $a$ and the probability decrease for larger prime. For example, $19,269,440297$ are Wieferich to the base $(3,13)$ all of order $2$, and there is no more for prime upto $10^8$. It is however easy to construct artificial bases for which a small prime is Wieferich of high order. For example, the expression $(3^k+t)^2-(3^k-t)^2=4t3^k$ implies $3$ is Wieferich to the base $(3^k-t,3^k+t)$ of order at least $k$.

 Wieferich prime of order $\ge 3$ are even rarer as we expect there are only finitely many such primes for each fixed base, We computed for primes upto $10^6$ for the $3043$ bases $(a,b)$ with $1 \le a <b \le 100$ with $gcd(a,b)=1$, there are $61$ base pairs with Wieferich primes  of order $ \ge 4$ and all the primes occuring are small $\le 17$, and only two base pair $(38,41),(3,79)$ where $5$ is Wieferich of order $5$.


\begin{theorem} (Higher order Wieferich criterion for FLT) Let $n$ be an odd prime and $x,y,z$ be relatively prime positive integers satisfying $x^n+y^n=z^n$. Then,
any odd prime $p \mid x$ is Wieferich to the base $(y,z)$ of order at least $\ge n -1$. Any odd prime $p \mid y$ is Wieferish to the base $(x,z)$ of order $\ge n-1$. Any odd prime $p \mid z$ is Wieferish to the base $(x,y)$ of order $\ge n-1$. For the "first case" where $n \nmid xyz$, we can increase the lower bound to $n$. Similarly for $p=2$, if $2|z$, $\nu_2(x+y) \ge n$. If $2|x$,
 $\nu_2(z-y) \ge n$. If $2|y$, $\nu_2(z-x) \ge n$.  If a $p$ dividing $x$ actually divides $z-y$, then we also have $\nu_p(z-y) \ge n-1$. Similarly for the other two cases.

\end{theorem}
\begin{proof} Let $p \mid y$ , then $p \nmid xz$, and $p^n|y^n=(x^n-z^n)$. By the identity,
$$\nu_p(x^{p-1}-z^{p-1})=\nu_p(x^n-z^n)-\nu_p(n) \ge n-\nu_p(n) \ge n-1.$$
If $p|z$, $p^n|(x^n+y^n)$ and we apply the plus identity. If $n \nmid xyz$, all $\nu_p(n)=0$.

Exactly one of $x,y,z$ is even and $\nu_2(n)=0$. If $2|z$, $2^n|(x^n+y^n)$, the plus identity holds with $\nu_2(x+y)=\nu_2(x^n+y^n) \ge n$. If $2|x$, $2^n|(z^n-y^n)$, and since $n$ is odd, we must have $yz=1$ mod $4$ so the minus identity holds, and $\nu_2(z-y)=\nu_2(z^n-y^n) \ge n$.

Among the primes dividing $x$, there must be some which also divide $z-y$ since $p|x^n=(z^n-y^n)=(z-y)(z^{n-1}+...+y^{n-1})$. By Lemma 1.3,  we have $\nu_p(z-y)=\nu_p(z^{p-1}-y^{p-1}) \ge n-1$. 
\end{proof}

The results still holds for $n=2$ where we have solution to verify. It state that if $x^2+y^2=z^2$ and $p$ is an odd prime dividing $xyz$ then $p$ is Wieferich to a base pair (complementary to the one it divides). This seems interesting since we know how to generate all Pythagorian tuples. We have for example $38399^2+2042040^2=2042401^2$, the $z$-term $p=2042401$ turn out to be a prime so we have
$2042401^2|(2042040^{2042400}-38399^{2042400})$.

The condition in Theorem 1.4 are so strong that it gives us very convincing intuition why FLT should fail for even moderately large $n$
with increasing improbability  for larger $n$. The natural question is how can we turn this into a proof for all sufficiently large $n$ ?


\subsection{Relevance to the $(a^n-1,1,a^n)$ abc tuple}  A weak form of the ABC conjecture applied to the tuple $(a^n-1)+1=a^n$ states essentially that the average power of $a^n-1$, $X_a:=\{ \frac{\log (a^n-1)}{\log( rad(a^n-1))}_{n \ge 1} \}$ is bounded. It seems we can't (?) even prove this very special case for a fixed $a=2$ say, and one key problem seems to be that we can't give a upper bound to the power that a prime can divide $a^n-1$. Using our identity , we can control the $\nu_p(a^n-1)-\nu_p(n)$ if we know a bound on the RHS which is independent of $n$. This leads us to defne for any integer $a>1$, its largest  Wieferich-order to be
$$W_a:=\sup_{p} \nu_p(a^{p-1}-1),$$ where the $sup$ is taken over all primes, and we  shall assume $W_a$ is bounded, which seems plausible. For infinitely many $a$ we don't even know if $W_a>1$. It is not known if $W_2>2$.  On the other hand, it is a century old theorem of Meyer \cite{Me} that for a fixed prime $p$, and an integer
$r>1$, and  $1 \le s \le r$, there exist $1<a <p^r$ with $\nu_p(a)=s$, so that $W_a$ can be  arbitrary large as $a$ varies, but this does not seem to contradict $W_a< \infty$ for any fixed $a$. We then have a very crude uniform bound $(a^n-1) < rad(a^n-1)^{\nu_p(n) +W_a}$,so that we have $X_a$ is bounded by $1+W_a$ over squarefree $n$.

We can refine the usual splitting of the factors of $a^n-1=m_1m_2$ where $m_1$ is squarefree and $m_2$ is the powerful part. Then we know $gcd(m_1,n)=1$ and $m_1$ is divisible by only non Wieferich primes. We also write $m_2=m_{2,1}m_{2,2}$
where $m_{2,1}$ is the part where  the prime divisor also divide $n$, then $m_{2,2}$ is divisible by only Wieferich primes. Also we split $m_{2,1}=m_Nm_W$ into parts divisible by non Wieferich and  Wieferich primes, we then have $m_N \le n.rad(n)$.
So we have $m=m_1m_Nm_Wm_{2,2}$  where the powerful part $m_W,m_{2,2}$ are divisible only by Wieferich primes which are rare, and
$m_N$ is bounded by a small quantity (compared to $m$). The largest part come from the squarefree part $m_1$ divisible only by non Wiefreich prime and there are known lower bound for $m_1$ in term of $m$. This gives us the intuition why $rad(m)$ is not much smaller than $m$,
as predicted by ABC.

\bibliographystyle{amsplain}

\end{document}